\documentclass[11pt,a4paper]{amsart}

\usepackage{amssymb,mathrsfs,enumitem,amsmath,amsthm,bbold}
\usepackage[mathscr]{euscript}
\usepackage{color}

\usepackage[all]{xy}
\definecolor{Chocolat}{rgb}{0.36, 0.2, 0.09}
\definecolor{BleuTresFonce}{rgb}{0.215, 0.215, 0.36}
\definecolor{EgyptianBlue}{rgb}{0.06, 0.2, 0.65}

\usepackage[OT2,OT1]{fontenc}

\usepackage[backend=bibtex,
    sorting=anyt,
    isbn=false,
    url=false,
    doi=false,
    maxbibnames=100
    ]{biblatex}
\usepackage{fourier}

\usepackage[colorlinks,final,hyperindex]{hyperref}
\hypersetup{citecolor=BleuTresFonce, urlcolor=EgyptianBlue, linkcolor=Chocolat}

\newtheorem{theorem}{Theorem}[section]

\newtheorem{proposition}[theorem]{Proposition}

\theoremstyle{definition}

\DeclareMathAlphabet{\pazocal}{OMS}{zplm}{m}{n}

\def\calO{\pazocal{O}}

\def\d{\mathrm{d}}

\DeclareMathAlphabet{\mathbbold}{U}{bbold}{m}{n}

\def\k{\mathbbold{k}}

\begin{document}

\title[A Nielsen--Schreier variety of algebras without the PBW property]{A Nielsen--Schreier variety of algebras\\ without the Poincaré--Birkhoff-Witt property}
\author{Vladimir Dotsenko}
\address{Institut de Recherche Math\'ematique Avanc\'ee, UMR 7501, Universit\'e de Strasbourg et CNRS, 7 rue Ren\'e-Descartes, 67000 Strasbourg, France}
\address{Institute of Mathematics and Mathematical Modeling, Pushkin St. 125, 050010 Almaty, Kazakhstan}
\email{vdotsenko@unistra.fr}

\author{Bekzat Zhakhayev}
\address{
SDU University, Abylaikhan St. 1/1,  040900, Kaskelen, Kazakhstan}
\address{Institute of Mathematics and Mathematical Modeling, Pushkin St. 125, 050010 Almaty, Kazakhstan}
\email{bekzat.kopzhasar@gmail.com}

\begin{abstract}
We prove the Nielsen--Schreier property for the variety of algebras defined by the identity $x(x^2)^2=(x^2)^2x$: every subalgebra of every free algebra in this variety is itself free. We also show that this variety of algebras does not have the Poincaré--Birkhoff--Witt property for universal multiplicative enveloping algebras. Our strategy of proof actually leads to infinitely many new varieties of non-associative algebras with the same behaviour. This offers new evidence supporting a conjecture of the first author and Umirbaev suggesting that the Nielsen--Schreier property over a field of zero characteristic is equivalent to freeness of universal multiplicative enveloping algebras of free algebras.
\end{abstract}

\maketitle

\section{Introduction}

All algebras satisfying a given system of identities are said to form a Nielsen--Schreier variety of algebras if every subalgebra of every free algebra of this type is free. Five Nielsen--Schreier varieties of algebras with one bilinear operation have been known from classical results in ring theory going back to the 1950s: all non-associative algebras (Kurosh \cite{MR0020986}), all Lie algebras (Shirshov and Witt \cite{MR0059892,MR77525}), all commutative or anticommutative non-associative algebras (Shirshov \cite{MR0062112}), and the trivial example of all algebras with zero product. In the 1990s, Umirbaev \cite{MR1302528} made an impressive breakthrough: he used universal multiplicative enveloping algebras and modules of Kähler differentials to establish a necessary and sufficient condition for a variety of algebras over a field of zero characteristic to have the Nielsen--Schreier property. However, the condition he gave is generally very hard to check, and consequently this criterion has seen rather limited use. 

Several years ago, the first author and Umirbaev \cite{MR4675074} used the theory of operads to establish an effective sufficient condition for a variety of algebras over a field of zero characteristic to have the Nielsen--Schreier property. At the time, it was not clear whether this condition was also necessary. The goal of this paper is to show that it is not the case. Our strategy relies on an observation that the condition of \cite[Th.~4.1]{MR4675074} implies, in particular, that our variety of algebras has the Poincaré--Birkhoff--Witt property for universal multiplicative enveloping algebras \cite{MR4381941}, and so it is enough to exhibit one Nielsen--Schreier variety that does not have the latter property. 

Our example is the variety $\mathfrak{M}$ of non-associative algebras defined by the rather innocent identity
 \[
x(x^2)^2=(x^2)^2x .    
 \]
In particular, we establish the validity for this variety of the property that is prominently featured in both \cite{MR1302528} and \cite{MR4675074}: the universal multiplicative enveloping algebra of every free $\mathfrak{M}$-algebra is free. Overall, our work should be viewed as supporting evidence in favour of \cite[Conjecture 4.8]{MR4675074}: we believe that the Nielsen--Schreier property over a field of zero characteristic is in fact equivalent to freeness of universal multiplicative enveloping algebras of free algebras. 

\subsection*{Conventions}

All vector spaces in this paper are defined over a field $\k$ of zero characteristic. We do not burden this short note with extensive recollections: the reader is expected to have certain knowledge of the language of algebraic operads \cite{MR2954392}, Gröbner bases for operads \cite{MR3642294}, and universal multiplicative enveloping algebras and Kähler differentials of algebras over operads \cite{MR2494775}. A reader in need of a short summary of how to use all those objects for the purposes of ring theory is invited to consult a concise introduction in \cite[Sec.~2]{MR4675074}.

\section{The Nielsen--Schreier property of the variety \texorpdfstring{$\mathfrak{M}$}{M}}

Our main result is the following theorem alluded to in the introduction.

\begin{theorem}\label{th:main}
Consider the variety $\mathfrak{M}$ of all non-associative algebras satisfying the identity
 \[
x(x^2)^2=(x^2)^2x.     
 \]
This variety is Nielsen--Schreier, and it does not have the Poincaré--Birkhoff--Witt property for universal multiplicative enveloping algebras. 
\end{theorem}

Our strategy is to apply the following result proved by the first author and Umirbaev, generalising the clever trick that Shirshov used in \cite{MR0059892} to establish that subalgebras of free Lie algebras are free. 

\begin{proposition}[{\cite[Lemma 4.6]{MR4675074}}]
Let $H_n$ be the kernel of the homomorphism from the free algebra $F_n:=F_\mathfrak{M}(y,x_1,\ldots,x_n)$ to the one-dimensional vector space spanned by $y$, viewed as an $\mathfrak{M}$-algebra with zero structure operations. If the $\mathfrak{M}$-algebra $H_n$ is free for each $n\ge 0$, the variety $\mathfrak{M}$ has the Nielsen--Schreier property.
\end{proposition}

We shall prove the freeness of the $\mathfrak{M}$-algebra $H_n$ in several steps. 

\begin{proposition}\label{prop:UEfree}
For every free $\mathfrak{M}$-algebra $A$, its universal multiplicative enveloping algebra $U_\mathfrak{M}(A)$ is a free associative algebra.
\end{proposition}

\begin{proof}
Note that the identity $x(x^2)^2=(x^2)^2x$ is equivalent, over a field of zero characteristic, to the multilinear identity 
 \[
\sum_{\sigma\in S_5} a_{\sigma(1)}((a_{\sigma(2)}a_{\sigma(3)})(a_{\sigma(4)}a_{\sigma(5)}))=\sum_{\sigma\in S_5} ((a_{\sigma(1)}a_{\sigma(2)})(a_{\sigma(3)}a_{\sigma(4)}))a_{\sigma(5)}.    
 \]
It will be convenient to consider a different description of our variety, using, instead of one binary operation $a,b\mapsto ab$, two operations
 \[
a_1\circ a_2=\frac12(a_1a_2+a_2a_1), \quad a_1\bullet a_2=\frac12(a_1a_2-a_2a_1),
 \]
the first of which is commutative, and the second of which is anticommutative. In terms of these operations, our identity becomes
 \[
\sum_{\sigma\in S_5} a_{\sigma(1)}\bullet((a_{\sigma(2)}\circ a_{\sigma(3)})\circ (a_{\sigma(4)}\circ a_{\sigma(5)}))=0.
 \]

In \cite[Lemma~4.2]{MR4675074}, it is established that universal multiplicative enveloping algebras of $\mathfrak{M}$-algebras are free as associative algebras if and only if for the reduced Gr\"obner basis of the corresponding shuffle operad $\calO^f$ with respect to the reverse graded path-lexicographic ordering, each leading term has the minimal leaf directly connected to the root.

If we consider the reverse graded path-lexicographic ordering (say, for a choice of a total order of generators given by $\circ<\bullet$), the leading term of our identity is the element 
 \[
a_{1}\bullet((a_{2}\circ a_{3})\circ (a_{4}\circ a_{5})),     
 \]
or, in the form of a tree monomial,
 
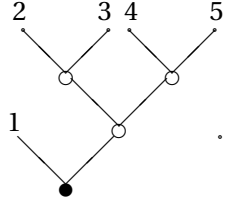
\[
\begin{picture}(40,80)
\put(0,-2){\circle*{5}} \put(58,18){\circle{1}}
\put(-16,58){\circle{1}} \put(16,58){\circle{1}}
\put(24,58){\circle{1}} \put(56,58){\circle{1}}
\put(20,20){\circle{5}} 
\put(0,40){\circle{5}} \put(40,40){\circle{5}} 
\put(0,0){\line(-1,1){18}} \put(0,0){\line(1,1){18}}
\put(20,22){\line(-1,1){18}} \put(20,22){\line(1,1){18}}
\put(0,42){\line(-1,1){16}} \put(0,42){\line(1,1){16}}
\put(40,42){\line(-1,1){16}} \put(40,42){\line(1,1){16}}
\put(-20,62){\text{2}} \put(12,62){\text{3}}
\put(22,62){\text{4}} \put(54,62){\text{5}} \put(-22,20){\text{1}}
\end{picture} 
 \]
Indeed, this element has the shortest path to the root to the leaf labelled $1$, which makes it larger than any other element with the leaf labelled $1$ not adjacent to the root, and then among the elements which have the leaf labelled $1$ adjacent to the root, the sequence of leaf labels in planar order is the smallest lexicographically. 

By a direct inspection, this element has no self-overlaps, since the root vertex is the only vertex carrying the label $\bullet$. Thus, the defining  relation of the shuffle operad $\calO^f$ forms a Gröbner basis for our ordering. Its leading term has the minimal leaf directly connected to the root, so our assertion follows.
\end{proof}

Our next step is as follows. 

\begin{proposition}
For every $n\ge 0$, the algebra $U_{\mathfrak{M}}(F_n)$ is a free left $U_{\mathfrak{M}}(H_n)$-module.
\end{proposition}

\begin{proof}
Let us choose a homogeneous basis $\mathsf{B}_n$ of $F_n$ for which all generators $y,x_1,\ldots,x_n$ are basis elements, and, additionally, if $b$ and $c$ are elements of $\mathsf{B}_n$, then $b\circ c=c\circ b$ is also an element of $\mathsf{B}_n$. To see that the latter property may be assured, we note that the defining identity of our variety is a linear dependency between elements whose top level operation is $\bullet$; therefore, all elements of the form $b\circ c$ are linearly independent from other elements and between each other as long as $b$ and $c$ are taken from the basis, and we may include all such elements when constructing a basis inductively degree by degree. By the virtue of our construction, the set $\mathsf{B}_n$ contains $y$ as an element, and we denote $\mathsf{B}_n':=\mathsf{B}_n\setminus\{y\}$; clearly, $\mathsf{B}_n'$ is a homogeneous basis of the algebra $H_n$.

The universal multiplicative enveloping algebra $U_{\mathfrak{M}}(F_n)$ is generated by the (countable) set 
 \[
X_n=\{\alpha_b, \beta_b \colon b\in \mathsf{B}_n\}. 
 \]
Relations between these generators are obtained as follows: we take the defining identity
 \[
\sum_{\sigma\in S_5} a_{\sigma(1)}\bullet((a_{\sigma(2)}\circ a_{\sigma(3)})\circ (a_{\sigma(4)}\circ a_{\sigma(5)}))=0
 \]
of our operad, declare one of the arguments (say, $a_5$, since the identity is symmetric in its arguments) ``special'', and substitute instead of the other arguments arbitrary elements of $F_n$ (say, $b_1$, $b_2$, $b_3$, $b_4$); this relation then becomes a relation of $U_\mathfrak{M}(F_n)$, if we write down what operations are applied to the special argument. For instance, 
 \[
a_{1}\bullet((a_{2}\circ a_{3})\circ (a_{4}\circ a_{5})) \leadsto   \beta_{b_1}\alpha_{b_2\circ b_3}\alpha_{b_4}   
 \]
and 
 \[
a_{5}\bullet ((a_{1}\circ a_{3})\circ (a_{2}\circ a_{4})) \leadsto -\beta_{(b_{1}\circ b_{3})\circ (b_{2}\circ b_{4})}.
 \]
 
Let us consider the total well ordering of $X_n$ defined as follows. We first choose a total well ordering of the set
 \[
X_n'=\{\alpha_b, \beta_b \colon b\in \mathsf{B}_n'\}, 
 \]
and then declare that each element from $X_n'$ is less than $\alpha_y$, which in turn is less than $\beta_y$. We then consider the degree-lexicographic ordering of monomials in the free associative algebra generated by $X_n$.
 
Note that if in the monomial
 \[
a_{\sigma(1)}\bullet((a_{\sigma(2)}\circ a_{\sigma(3)})\circ (a_{\sigma(4)}\circ a_{\sigma(5)}))
 \]
we have $\sigma(1)=5$, this monomial corresponds to $\beta_{(b_{i_1}\circ b_{i_3})\circ (b_{i_2}\circ b_{i_4})}$ for some subscripts $i_1,i_2,i_3,i_4$, and if we have $\sigma(i)=5$ for some $i\ne 1$, this monomial corresponds to $\beta_{b_{i_1}}\alpha_{b_{i_2}\circ b_{i_3}}\alpha_{b_{i_4}}$ for some subscripts $i_1,i_2,i_3,i_4$. Thus, for the ordering that we defined, one of the latter monomials
 \[
\beta_{b_{i_1}}\alpha_{b_{i_2}\circ b_{i_3}}\alpha_{b_{i_4}} 
 \]
is the leading term of the corresponding relation. Moreover, if at least one of the elements $b_{i_1},b_{i_2},b_{i_3},b_{i_4}$ is equal to $y$, then our definition of the ordering implies that the leading monomial is of the form given above with $b_{i_1}=y$. Let us note that these leading terms do not form any overlaps with each other: each of them has length three, begins with a $\beta$-letter, and its last two letters are $\alpha$-letters, and hence no proper suffix of one of them can be a proper prefix of another. Hence our relations form a Gr\"obner basis $G_n$ of relations of $U_\mathfrak{M}(F_n)$. Let us denote by $N_n$ the set of normal forms with respect to $G_n$ in the free associative algebra generated by $X_n$; those normal forms constitute a vector space basis of $U_{\mathfrak{M}}(F_n)$. 
 
Note that the universal multiplicative enveloping algebra $U_\mathfrak{M}(H_n)$ is generated by the set $X_n'$. Moreover, the subset of our relations of $U_{\mathfrak{M}}(F_n)$, where we take all $b_1,b_2,b_3,b_4$ to be different from $y$, is a full set of defining relations of $U_\mathfrak{M}(H_n)$; by the same argument, these relations  form a Gröbner basis $G_n'$ of relations of $U_\mathfrak{M}(H_n)$. Let us denote by $N_n'$ the set of normal forms with respect to $G_n'$  in the free associative algebra generated by $X_n'$; those normal forms constitute a vector space basis of $U_{\mathfrak{M}}(H_n)$.

Finally, let us denote by $N_n^y$ the subset of $N_n$ consisting of the empty word and of all normal forms whose first letter is $\alpha_y$ or $\beta_y$.
Let us show that there is a bijection 
 \[
N_n'\times N_n^y \to N_n
 \]
sending a pair $(w',w'')$ to the product $w'w''$. Indeed, each normal form $w\in N_n$ admits an obvious (and unique) factorization $w=w'w''$, where $w'\in N_n'$ and $w''\in N_n^y$, and conversely, if $w'\in N_n'$ and $w''\in N_n^y$, then $w'w''$ is a normal form, since if it were divisible by a leading term of a relation, this leading term would have its first letter in $X_n'$ and one of the following letters $\alpha_y$ or $\beta_y$, which is impossible.

Passing to vector spaces spanned by normal forms, we obtain an isomorphism
 \[
 \k N_n'\otimes \k N_n^y\cong \k N_n ,
 \]
which immediately implies that $U_{\mathfrak{M}}(F_n)$ is a free left $U_{\mathfrak{M}}(H_n)$-module. 
\end{proof}

\begin{proposition}\label{prop:HnFree}
For every $n\ge 0$, the $\mathfrak{M}$-algebra $H_n$ is free.
\end{proposition}

\begin{proof}
The argument proving our assertion is in fact completely analogous to the proof of \cite[Lemma 5]{MR1302528}. For the sake of completeness, we give a detailed proof. 

Let $Y_n\subset \mathsf{B}_n'$ be a subset of the basis of the algebra $H_n$ which is a minimal system of generators of that algebra. Suppose that this system of generators is not free, so that there exist a positive integer $m$, elements $h_1,\ldots,h_m\in Y_n$, and an element $f\in F_\mathfrak{M}(z_1,\ldots,z_m)$ for which
 \[
f(h_1,\ldots,h_m)=0. 
 \]
Moreover, we may choose $f$ to be homogeneous of the smallest possible degree. Let us apply to this relation the universal derivation
 \[
\d\colon F_n\to \Omega^1_{\mathfrak{M}}(F_n), 
 \]
obtaining
 \[
\sum_{k=1}^m \frac{\partial f}{\partial z_k}|_{z_1\leftarrow h_1,\ldots,z_m\leftarrow h_m }\d(h_k)=0.
 \]
Since $f\ne 0$, at least one of the elements $\frac{\partial f}{\partial z_k}$ is nonzero; moreover, the degree of $\frac{\partial f}{\partial z_k}$ is less than the degree of $f$, and hence $\frac{\partial f}{\partial z_k}\ne 0$ implies $\frac{\partial f}{\partial z_k}|_{z_1\leftarrow h_1,\ldots,z_m\leftarrow h_m }\ne 0$ due to our assumption on the minimality of the degree of $f$. It follows that the elements $\d(h_1)$, \ldots, $\d(h_m)$ of the $U_{\mathfrak{M}}(H_n)$-module $\Omega^1_{\mathfrak{M}}(F_n)$ are linearly dependent.  

Since $U_{\mathfrak{M}}(F_n)$ is free as a left $U_{\mathfrak{M}}(H_n)$-module, the module $\Omega^1_{\mathfrak{M}}(F_n)$, which is manifestly free as a left $U_{\mathfrak{M}}(F_n)$-module, is also free as a left $U_{\mathfrak{M}}(H_n)$-module.

Moreover, since the associative algebra $U_{\mathfrak{M}}(F_n)$ is free as an associative algebra, we may use \cite[Th.~6.6]{MR800091} to conclude that freeness of $U_{\mathfrak{M}}(F_n)$ as a $U_{\mathfrak{M}}(H_n)$-module implies freeness of $U_{\mathfrak{M}}(H_n)$  as an associative algebra. (To be precise, \emph{op. cit.} requires $U_{\mathfrak{M}}(H_n)$ to be homogeneous; to ensure that, we may consider the grading of $U_{\mathfrak{M}}(F_n)$ induced from the natural grading of $F_n$.) For free algebras, submodules of free modules are free; it follows that one of the elements $\d(h_1)$, \ldots, $\d(h_m)$ is a $U_{\mathfrak{M}}(H_n)$-linear combination of the others. Without loss of generality, we have 
 \[
\d(h_m)=\sum_{k=1}^{m-1}c_k\d(h_k) 
 \]
with $c_k\in U_{\mathfrak{M}}(H_n)$,  
implying, for each $t\in\{x_1,\ldots,x_n,y\}$,
 \[
\frac{\partial h_m}{\partial t}=\sum_{k=1}^{m-1}c_k\frac{\partial h_k}{\partial t}.
 \]
Using the ``Euler formula''
 \[
\sum_{t\in\{x_1,\ldots,x_n,y\}}\frac{\partial h}{\partial t}t=\deg(h)h,
 \]
valid for all $h\in F_n$, we see that 
 \[
\deg(h_m)h_m=\sum_{k=1}^{m-1}\deg(h_k)c_kh_k. 
 \]
By degree reasons, each $c_k$ only involves elements of $H_n$ of degree strictly smaller than $\deg(h_m)$, hence belongin to the subalgebra generated by $Y_n\setminus\{h_m\}$, contradicting the minimality of $Y_n$ and completing the proof.
\end{proof}

To complete the proof of the main result of the paper, we need to establish one more result. 

\begin{proposition}\label{prop:PBWiff}
A variety $\mathfrak{M}$ whose structure operations all have the same arity has the Poincaré--Birkhoff--Witt property for universal multiplicative enveloping algebras if and only if for the reduced Gr\"obner basis of the corresponding shuffle operad $\calO^f$ with respect to the graded path-lexicographic ordering, each leading term is a left comb. 
\end{proposition}

\begin{proof}
According to \cite[Th.~3.4]{MR4381941}, a variety $\mathfrak{M}$ has the Poincaré--Birkhoff--Witt property for universal multiplicative enveloping algebras if and only if the species derivative $\partial(\calO)$ is free as a right $\calO$-module. 

Furthermore, it is established in  \cite[Th.~5.16]{MR4381941} that if for the reduced Gr\"obner basis of the corresponding shuffle operad $\calO^f$ with respect to the graded path-lexicographic ordering, each leading term is a left comb, then $\partial(\calO)$ is free as a right $\calO$-module, and hence $\mathfrak{M}$ has the Poincaré--Birkhoff--Witt property for universal multiplicative enveloping algebras. 

Suppose now that $\mathfrak{M}$ has the Poincaré--Birkhoff--Witt property for universal multiplicative enveloping algebras, which we already know to be equivalent to the freeness of the right $\calO$-module $\partial(\calO)$. 

We shall use the following observation twice: if the leading term with respect to the graded path-lexicographic ordering of some element of the free operad generated by several operations of the same arity is not a left comb, then none of the monomials appearing in this element is a left comb. Indeed, if that were the case, there would exist a non-leading monomial that is manifestly larger than the leading monomial with respect to the graded path-lexicographic ordering. 

Let us first note that it is always true that a minimal system of generators of $\partial(\calO)$ as a right $\calO$-module consists of those normal forms with respect to the graded path-lexicographic ordering that are left combs. Indeed, each normal form for $\calO$ is obtained by right module action avoiding the minimal leaf from the left comb normal forms, so these normal forms are a generating system. Suppose that this generating system of $\partial(\calO)$ is not minimal, that is one of these normal forms, say $m$ is equal to a right $\calO$-module combination of the others which avoids the minimal leaf. The element $f$ that is equal to the difference of $m$ and the said right $\calO$-module combination belongs to the ideal of relations of $\calO$, and by the above observation, $m$ must be its leading term. But then $m$ cannot be a normal form with respect to a Gröbner basis, a contradiction. 

To conclude, we suppose that our Gröbner basis contains an element $g$ whose leading term is not a left comb. By the above observation, none of the terms appearing in $g$ is a left comb, and hence $g$ is a nontrivial right $\calO$-module dependency between the minimal generators of the module $\partial(\calO)$, which we assumed free. The contradiction completes the proof.  
\end{proof}

\begin{proof}[Proof of Theorem \ref{th:main}]
It is established in  \cite[Lemma 4.6]{MR4675074} that if the $\mathfrak{M}$-algebra $H_n$ is free for each $n\ge 0$, the variety $\mathfrak{M}$ has the Nielsen--Schreier property, so Proposition \ref{prop:HnFree} implies that our variety is Nielsen--Schreier. 

We shall use Proposition \ref{prop:PBWiff} to prove that our variety does not have the Poincaré--Birkhoff--Witt property for universal multiplicative enveloping algebras.  In our case, the leading term of the defining relation of the corresponding shuffle operad is 
 
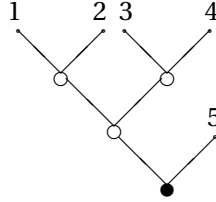
\[
\begin{picture}(40,80)
\put(40,-2){\circle*{5}} \put(58,18){\circle{1}}
\put(-16,58){\circle{1}} \put(16,58){\circle{1}}
\put(24,58){\circle{1}} \put(56,58){\circle{1}}
\put(20,20){\circle{5}} 
\put(0,40){\circle{5}} \put(40,40){\circle{5}} 
\put(40,0){\line(-1,1){18}} \put(40,0){\line(1,1){18}}
\put(20,22){\line(-1,1){18}} \put(20,22){\line(1,1){18}}
\put(0,42){\line(-1,1){16}} \put(0,42){\line(1,1){16}}
\put(40,42){\line(-1,1){16}} \put(40,42){\line(1,1){16}}
\put(55,22){\text{5}} \put(-20,62){\text{1}} \put(12,62){\text{2\textbf{}}}
\put(22,62){\text{3}} \put(54,62){\text{4}} 
\end{picture} 
 \]
This element does not form any self-overlaps, and hence the given defining relation forms a Gröbner basis. Its leading term is manifestly not a left comb, and therefore the variety $\mathfrak{M}$
does not have the Poincaré--Birkhoff--Witt property for universal multiplicative enveloping algebras.  
\end{proof}

\section{Conclusion}

We reiterate that, according to \cite[Lemma 4.3]{MR4675074}, every variety of algebras to which the criterion of \cite{MR4675074} applies has the Poincaré--Birkhoff--Witt property for universal multiplicative enveloping algebras, and hence our example is outside the scope of that criterion. An almost identical argument proves that any identity of the form
 \[
x\bullet(m_{1}\circ m_{2})=0, 
 \]
where $m_1$ and $m_2$ are monomials in the variable $x$ obtained by iterations of the symmetrized product $\circ$, defines a Nielsen--Schreier variety that does not have the Poincaré--Birkhoff--Witt property for universal multiplicative enveloping algebras. Indeed, the leading term again has the minimal leaf adjacent to the root for the reverse graded  path-lexicographic ordering, while for the graded path-lexicographic ordering the leading term is not a left comb; moreover, the defining monomial has no self-overlaps because the root is the only vertex labelled by $\bullet$.

Thus, one may be hopeful that the Nielsen--Schreier property over a field of zero characteristic is in fact equivalent to freeness of universal multiplicative enveloping algebras of free algebras.

\section*{Funding }
This research was funded by the Science Committee of the Ministry of Science and Higher Education of the Republic of Kazakhstan (Grant No. BR 28713025). 

\section*{Acknowledgements} 
The first author wishes to acknowledge motivating discussions with Ivan Shestakov who asked him whether there is a relationship between the PBW property for universal multiplicative enveloping algebras and the Nielsen--Schreier property, and useful conversations with Ualbai Umirbaev during early stages of work on this paper.

\printbibliography

\end{document}